\documentclass[11pt]{article}
\usepackage{amsmath,amsthm, hyperref, amsfonts}
\usepackage{amssymb}
\usepackage[margin=1in]{geometry}
\usepackage{epsfig}
\usepackage{color}       
\usepackage{comment}

\newtheorem{theorem}{Theorem}[section]
\newtheorem{lemma}[theorem]{Lemma}

\newtheorem{corollary}[theorem]{Corollary}
\newtheorem{observation}[theorem]{Observation}

\renewcommand{\a}{\alpha}

\newcommand{\ceil}[1]{\left \lceil #1 \right \rceil}
\newcommand{\of}[1]{\left( #1 \right)}
\newcommand{\set}[1]{\left\{ #1 \right\}}
\newcommand{\abs}[1]{\left| #1 \right|}

\newcommand{\tbf}[1]{\textbf{#1}}

\renewcommand{\a}{\alpha}
\renewcommand{\b}{\beta}

\newcommand{\up}{\textrm{up}}
\newcommand{\clecn}{\mathcal{C}_{\le cn}}
\newcommand{\clen}{\mathcal{C}_{\le n}}

\title{On the size Ramsey number of all cycles versus a path}

\author{ {\Large Deepak Bal} \thanks{\texttt{deepak.bal@montclair.edu} } \qquad
{\Large Ely Schudrich} \thanks{\texttt{schudriche1@montclair.edu}} 
\\ Department of Mathematical Sciences \\ Montclair State University \\ Montclair, NJ, USA
}
 
\date{}
\begin{document}
\maketitle

\begin{abstract}
We say $G\to (\mathcal{C}, P_n)$ if $G-E(F)$ contains an $n$-vertex path $P_n$ for any spanning forest $F\subset G$.
The size Ramsey number $\hat{R}(\mathcal{C}, P_n)$ is the smallest integer $m$ such that there exists a graph $G$ with $m$ edges for which $G\to  (\mathcal{C}, P_n)$. Dudek, Khoeini and Pra{\l}at proved that for sufficiently large $n$,
$2.0036n \le \hat{R}(\mathcal{C}, P_n)\le 31n$. In this note, we improve both the lower and upper bounds to $2.066n\le \hat{R}(\mathcal{C}, P_n)\le 5.25n+O(1).$ Our construction for the upper bound is completely different than the one considered by Dudek, Khoeini and Pra{\l}at. We also have a computer assisted proof of the upper bound $\hat{R}(\mathcal{C}, P_n)\le \frac{75}{19}n +O(1) < 3.947n $. 
\end{abstract}

\section{Introduction}

Let $\mathcal{F}$ be a family of graphs and let $H$ be a graph. We say that $G\to(\mathcal{F}, H)$ if every red/blue coloring of the edges of $G$ contains a monochromatic red copy of some graph from $\mathcal{F}$ or a monochromatic blue copy of $H$. The \emph{size Ramsey number} is defined as
\[\hat{R}(\mathcal{F}, H) = \min\set{ |E(G)|\,:\,G\to (\mathcal{F}, H)    }.\]
In the case where $\mathcal{F}=  \set{F}$, we will write $\hat{R}(F,H)$ for $\hat{R}(\mathcal{F},H)$ and we write $\hat{R}(H)$ for $\hat{R}(H,H)$.
To prove the upper bound $\hat{R}(\mathcal{F}, H) \le m$, one must prove the existence of a graph $G$ with $m$ edges such that $G\to (\mathcal{F}, H)$. To prove the lower bound $\hat{R}(\mathcal{F}, H) \ge m$, one must show that for every graph $G$ on $m-1$ edges, there is a 2 coloring which avoids both monochromatic graphs from $\mathcal{F}$ and $H$.

Let $P_n$ be the path on $n$ vertices. The size Ramsey number $\hat{R}(P_n)$ has been extensively studied, perhaps due to the fact that Erd\H{o}s \cite{Er} offered \$100  for a proof or disproof of $\hat{R}(P_n) = O(n)$. Beck answered the question \cite{B2}, showing that $\hat{R}(P_n) \le 900n$. After a series of improvements to the upper bound \cite{Bol, DP1,Let,DP2} and the lower bound \cite{B2, Bol, DP2, BD}, the state of the art is $(3.75 + o(1))n \le \hat{R}(P_n)\le 74n$ for $n$ sufficiently large. 
The size Ramsey number of $C_n$,  the cycle of length $n$, was first proven to be linear in $n$ by Haxell, Kohayakawa, and {\L}uczak \cite{HKL} with use of the sparse regularity lemma. A proof of this avoiding the use of regularity and providing explicit constants was given by Javadi, Khoeini, Omidi and Pokrovskiy \cite{JKOP}, who proved that $\hat{R}(C_n)\le 10^6cn$ where $c=843$ if $n$ is even and $c= 113482$ if $n$ is odd.   The proofs of  these upper bounds as well as the best known upper bounds for $\hat{R}(P_n)$ use random (regular) graphs as their construction. 

For any $c\in \mathbb{R}_+$, let $\mathcal{C}_{\le cn}$ be the family of all cycles of length at most $cn$ and let $\mathcal{C}$ be the family of all cycles. In \cite{DKP}, Dudek, Khoeini and Pra{\l}at initiated the study of $\hat{R}(\clecn, P_n)$ and $\hat{R}(\mathcal{C}, P_n)$. 
We remark that the parameter $\hat{R}(\mathcal{C}, P_n)$ is perhaps a natural one to study. If $G\to (\mathcal{C}, P_n)$, then $G$ contains a path of order $n$ after the removal of the edges of any spanning forest. 

Concerning lower bounds, first note that for any $c\in \mathbb{R}_+$, $\hat{R}(\clecn, P_n)\ge \hat{R}(\mathcal{C}, P_n) \ge 2(n-1)$. The first inequality follows from the fact that any coloring of a graph which avoids all cycles in red, clearly avoids all cycles of length at most $cn$ in red.  For the second inequality, take any (connected) graph on $2(n-1)-1$ edges (and at least $n$ vertices), color any spanning tree red, and note that there are not enough edges remaining to form a blue $P_n$. It is not immediately clear how one can move away from this trivial lower bound, but in \cite{DKP}, the authors managed to prove that for sufficiently large $n$ and any $c\in \mathbb{R}_+$, $\hat{R}(\clecn, P_n)\ge \hat{R}(\mathcal{C}, P_n) \ge 2.00365n$. 

For the upper bound, the authors of \cite{DKP}
use a random graph construction and techniques similar to those in \cite{DP1,  DP2, Let} to prove that 
\begin{align}\label{eq:dkpmain}
\hat{R}(\clecn, P_n) \le \begin{cases}
\frac{80\log (e/c)}{c}n & \textrm{for }c<1\\
31n & \textrm{for } c\ge 1
\end{cases}\end{align}
Note that as $c\to 0$, this upper bound tends to infinity. It is mentioned in \cite{DKP} that due to monotonicity ($m_1 \ge m_2 \implies \hat{R}(\mathcal{C}_{\le m_1}, P_n) \le \hat{R}(\mathcal{C}_{\le m_2},P_n)$), it is perhaps plausible that there is some decreasing function $\b(c)$, such that for each fixed $c>0$, $\hat{R}(\clecn, P_n) \sim \b(c)n$. They mention that the ``limiting case'' $c\to \infty$ corresponds to  $\hat{R}(\mathcal{C}, P_n)$ but they are only able to prove the upper bound $\hat{R}(\mathcal{C}, P_n) \le \hat{R}(\clen, P_n) \le 31n$.  

In this note, we show  that a significant improvement in the upper bound for $\hat{R}(\mathcal{C}, P_n)$ can be attained, not by considering the limit as $c$ grows large, but rather by considering very \emph{small} values of $c$. In fact, for our improvement, it is enough to  only consider red cycles of length $3,4$ or $5$. This fact may seem surprising given the behavior of the upper bound provided in \eqref{eq:dkpmain} as $c\to0$, but in light of the construction we provide, the surprise diminishes. 
Recall that for a graph $G=(V,E)$, the $k$th power, $G^k$ is a graph on vertex set $V$ in which  two vertices are adjacent if they are of distance at most $k$ in graph $G$. In our main theorem, we abandon random constructions altogether and show that a very structured graph, the third power of a path, suffices.

\begin{theorem}\label{thm:main}
Let $n\ge 2$ and let $N \ge \frac74 n  +10$. Then $P_{N}^3 \to (\mathcal{C}_{\le 5}, P_n)$. 
\end{theorem}
By monotonicity, this result improves the entire range of results stated in \eqref{eq:dkpmain}.

\begin{corollary}
For any $c \in \mathbb{R}^{+}$,  
\[\hat{R}(\mathcal{C}, P_n)\le \hat{R}(\mathcal{C}_{\le cn}, P_n) \le \hat{R}(\mathcal{C}_{\le 5}, P_n) \le \frac{21}{4}n + 27.\]
\end{corollary}
\begin{proof}
The first two inequalities follow from monotonicity. Let $N = \ceil{ \frac74 n + 10}$. Then \[|E(P_N^3)| = 3(N-3) + 2+1 = 3N-6 \le \frac{21}{4}n  +27.  \] 
\end{proof}

Making use of a lemma proved with a computer check (described in Section \ref{sec: comp}), we have the following improvement.

\begin{theorem}\label{thm:comp}
Let $n\ge 2$ and let $N \ge \frac{25}{19} n  + 43$. Then $P_N^3 \to (\mathcal{C}_{\le 8}, P_n)$. Thus we have the bound
\[\hat{R}(\mathcal{C}, P_n) \le \hat{R}(\mathcal{C}_{\le 8}, P_n) \le \frac{75}{19}n + O(1) < 3.947n +O(1).\]
\end{theorem}

We remark that one interesting fact about $P_N^3$ is that it is a maximal planar graph and is in fact an \emph{Apollonian network}. That is, it can be drawn by starting with a triangle in the plane and then repeatedly adding a new vertex inside of a current face and connecting it to each vertex of the containing face. Such a planar drawing is shown in Figure \ref{fig:apol}. 
\begin{figure}[h]
\centering
\includegraphics[scale=.6]{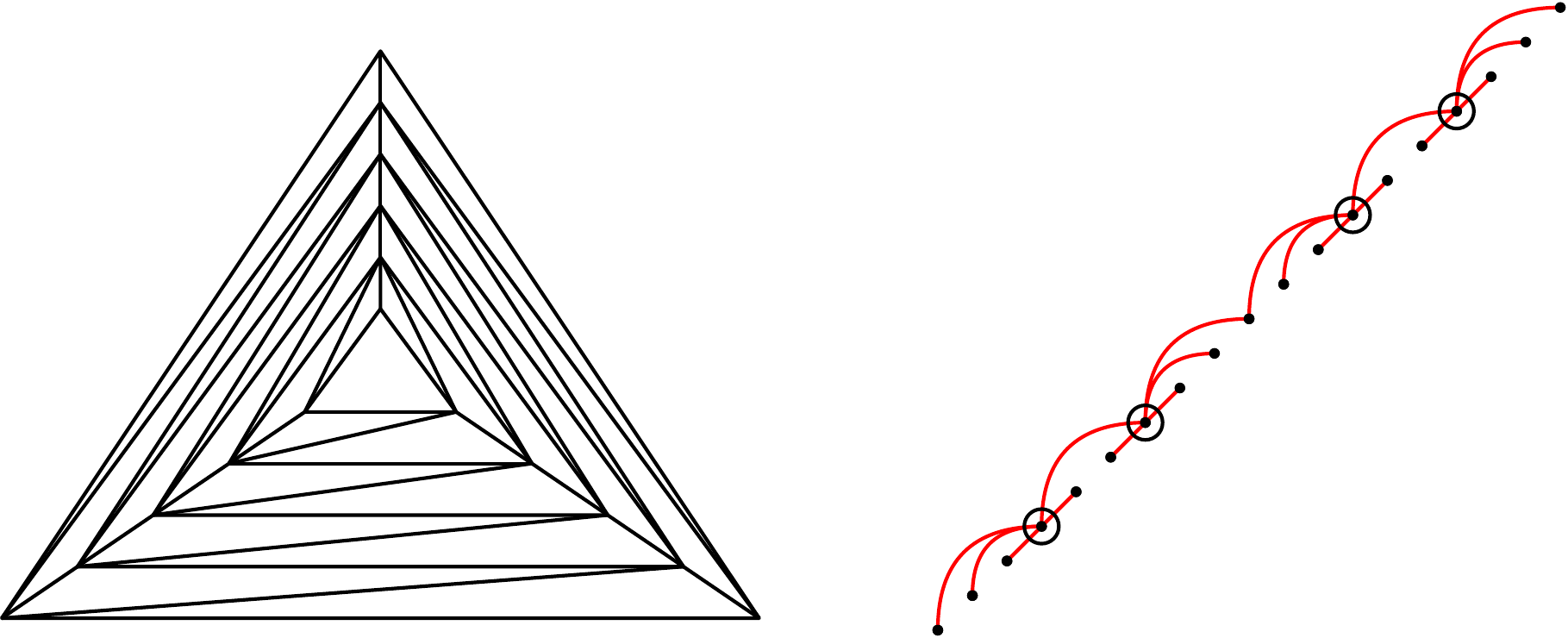}
\caption{On the left is a planar drawing of $P_N^3$. On the right is a spanning tree of $P_N^3$ whose removal leaves behind a path of density $\sim 7/9$.}\label{fig:apol}
\end{figure}

In this paper we also consider the lower bound. By improving upon the ideas in \cite{DKP}, we prove the following theorem.

\begin{theorem}\label{thm:lb}
Suppose $n$ is sufficiently large and $G$ is a graph with at most $\of{2+\frac{43}{651}}n - O(1)$ edges. Then there exists a red/blue coloring of $E(G)$ such that the red graph is acyclic and the blue graph contains no path of order $n$. Thus
\[2.066n <  \of{2+\frac{43}{651}}n - O(1) \le  \hat{R}(\mathcal{C}, P_n) \]
\end{theorem}

\section{Proof Idea and Notation}

\subsection{Upper Bound}\label{sec:ub}
Given an integer vertex set $[N] =  \set{1,2,\ldots, N}$, we call the path with  $i\sim (i+1)$ for all $i=1,\ldots, N-1$ a \emph{base path}.
Let $N \ge \frac74 n  +10$ and let $P:= P_N$ be the base path on vertex set $[N]$. Define $G: = P_N^3$. We will prove that every red/blue coloring of $E(G)$ with no red $C_3,C_4$ or $C_5$ contains a blue path of order at least $n$.

Suppose $Q$ is a base path on vertex set $\set{0,1,\ldots, \ell}$ and $H=  Q^3.$ The \emph{density} of a path $P$ in $H$ with endpoint $0$ is defined as
\[r(P) := \frac{|V(P) \cap \set{1,2,\ldots, \ell}|}{\ell}.\]
The following observation shows that one can ``stitch together'' paths while maintaining the density of the longer path.
\begin{observation}\label{obs:stitch}
Suppose $Q$ is a base path on vertex set $\set{0,1,\ldots, k, k+1,\ldots, k+\ell}$ and $H=Q^3$. Suppose that $P_1$ is a path in $H[\set{0,1,\ldots, k}]$ with endpoints 0 and $k$ and $r(P_1) = d_1$, and that $P_2$ is a path in $H[ \set{k, k+1, \ldots,k+\ell }  ]$ with endpoints $k$ and $k+\ell$ and $r(P_2) = d_2$. Then $P_1\cup P_2$ is a path in $H$ with endpoints 0 and $k+\ell$ and $r(P_1\cup P_2) \ge \min\set{d_1, d_2}$. 
\end{observation}
\begin{proof}
The fact that $P_1 \cup P_2$ forms a path in $H$ is obvious.  For the density, suppose $\hat{d} = \min\set{d_1,d_2}$. Then we have
\begin{align*} r(P_1\cup P_2) &=  \frac{    |  V(P_1\cup P_2) \cap \set{1,2,\ldots, k+\ell}|}{k+\ell}\\
&=  \frac{ |V(P_1)\cap \set{1,\ldots,k}|  + | V(P_2)\cap \set{k+1,\ldots, k+\ell } |   }{k+\ell}\\
&=\frac{d_1k +d_2\ell}{k+\ell} \ge \hat d.
\end{align*}

\end{proof}

Throughout the paper, we will make use of the underlying order of the vertex set of $G = P_N^3$.  Each vertex of $G$ in $\set{4,5,\ldots, N-3}$ has exactly 6 neighbors: $v\pm i$ where $i\in[3]$. For each vertex $v\in[N-3]$, we refer to the neighbors $v+i$, $i\in[3]$ as the \emph{up-neighbors} of $v$.
Given a red/blue (or $R$/$B$ for short) coloring of $E(G)$, for each vertex $v \in [N-3]$, we may associate an element of $\set{R,B}^3$ (i.e. a string of length 3 with entries from $\set{R,B}$) representing the colors assigned to the edges between $v$ and its up-neighbors.
We use the notation $\up(v)=c_1c_2c_3$ to mean that the edges $\set{v,v+1}, \set{v,v+2}, \set{v,v+3}$  are colored with $c_1, c_2, c_3$ respectively. As an illustration of this notation, we highlight one fact which we will use repeatedly without mention. If $G$ contains no red cycles, and  $\up(v) = RRR$, then vertices $v+1, v+2, v+3$ form a blue triangle (else there would be a red $C_3$). See Figure \ref{fig:up-rrr}.

\begin{figure}[h]
\centering
\includegraphics[scale=1]{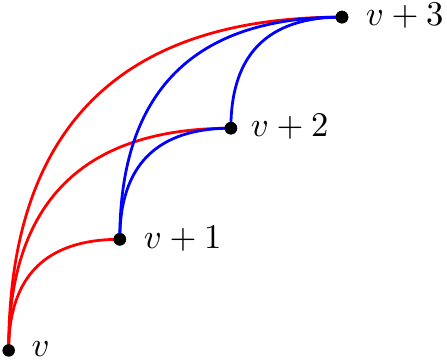}
\caption{Blue $C_3$ when $\up(v) = RRR$.}
\label{fig:up-rrr}
\end{figure}

The main idea of the proof is to suppose that $G$ has been $R/B$ colored such that there is no red cycle of length at most $5$ and to show that in this case, there must be a blue path of order at least $n$. We will find the long blue path by showing that starting at any vertex $v$ with $\up(v)\neq RRR$, one can find a blue path of density at least $4/7$ in the next 10 consecutive vertices with endpoints $v$ and $w$ where $\up(w)\neq RRR$. These short high density blue paths can then be stitched together as in Observation \ref{obs:stitch} to form the long blue path. The following lemma which is the main ingredient in our proof of Theorem \ref{thm:main}, says that the short high density paths can always be found.

\begin{lemma}\label{lem:mainlem}
Let $Q = P_{11}$ on vertex set $\set{0,1,\ldots 10}$ and let $H = Q^3$. Suppose that $H$ has been 2-colored with no red cycles from $\mathcal{C}_{\le5}$. Further suppose that in $H$, $\up(0)$ contains at least one $B$. Then there is a $k\in \set{1,\ldots 9}$ such that $H[\set{0,\ldots, k}]$ contains a blue path $P_B$ with endpoints $0$ and $k$ such that $up(k)$ contains at least one $B$ and
\[r(P_B):=\frac{\abs{V(P_B)\cap \set{1,\ldots, k}}}{k} \ge \frac{4}{7}.\]
\end{lemma}
With this lemma in hand (proved in Section \ref{sec:mainlempf}), we can prove the main theorem.

\begin{proof}[Proof of Theorem \ref{thm:main}]
Let $N\ge \frac74n + 10$, let $G=P_N^3$ and suppose that $E(G)$ has been 2-colored with red and blue such that there is no red cycle from $\mathcal{C}_{\le 5}$. It cannot be the case that vertices 1 and 2 both have 3 red up-neighbors. Hence we may apply Lemma \ref{lem:mainlem} starting at one of these vertices. We then repeatedly apply Lemma 2.2 to find an extension of the current blue path to another with density at least $4/7$ (by Observation \ref{obs:stitch}). We continue extending the blue path until we have found one, $P_B$, whose endpoint lies in $\set{N-9, \ldots, N }$ (if the last blue endpoint is smaller than $N-9$, then Lemma \ref{lem:mainlem} can be applied again). Then since $r(P_B) \ge 4/7$, we have
\[|V(P_B)| \ge \frac{4}{7}\cdot (N-11) +1 \ge n    \]
where we have used $N-11$ since $P_B$ may start at vertex 2 and the additional 1 accounts for the very first vertex of $P_B$. 
\end{proof}

The largest blue path density one could hope for in $P_N^3$ is $7/9$ since we may color the edges red in a repeating pattern as  indicated by Figure \ref{fig:apol} . At most 2 of the circled vertices may be used in a blue path (as endpoints) since they would have blue degree 1. Thus we have the following.

\begin{observation}
 The best upper bound that one could ever prove using the cube of a path is $\hat{R}(\mathcal{C}, P_n)\le\frac97n\cdot 3+O(1) \approx 3.857n+O(1)$ .
\end{observation}

\subsection{Lower Bound}
In order to improve the lower bound, we must show that every graph $G$ with at most $(2+\alpha)n$ edges contains a forest whose removal destroys all the paths of order $n$. One approach to accomplish this is to find a forest which contains many vertices of full degree (that is, vertices with the same degree in the forest as in the graph $G$). Such full degree vertices cannot be used in a blue path. This is the approach taken in \cite{DKP}. One snag is that it is not so simple to find such forests in graphs with unbounded degree. The proof of Theorem \ref{thm:lb} shows how to deal with high degree vertices and also gives an improved approach for bounded degree graphs than the one in \cite{DKP}.

\subsection{Notation and outline}
We use $N(v)$ to refer to the open neighborhood of vertex $v$. For two subsets $X,Y$ of vertices, we use $e(X,Y)$ to represent the number of edges with one endpoint in $X$ and one in $Y$. 
In Section 3, we deal with  a graph on vertex set $\set{0,1,\ldots 10}$ and since we do not refer to vertex 10 in the proof, we choose to omit commas when naming paths and cycles. For example the path $(0,1,3,4)$ will be denoted by $0134$ and the cycle on those same vertices will be denoted $(0134)$.

In Section \ref{sec:mainlempf} we prove Lemma \ref{lem:mainlem}. In Section \ref{sec: comp} we briefly describe the computer assisted improvement to Lemma \ref{lem:mainlem} which implies Theorem \ref{thm:comp}. In Section \ref{sec:lb} we prove Theorem \ref{thm:lb}.

\section{Proof of Main Lemma}\label{sec:mainlempf}
\subsection{A warm-up: density $1/3$}
In this subsection, to give a flavor of the proof to come, we prove a version of Lemma \ref{lem:mainlem}, replacing $4/7$ with $1/3$. We split into 7 cases depending on $\up(0)$. Note by assumption, we do not consider the case $\up(0) = RRR$. Our goal in each case is to find a blue path $P_B$ with density $r(P_B)\ge 1/3$ such that the non-$0$ endpoint, $k$, has a blue up-neighbor. We also note that at the end of this short warm-up  we will already have proved that $P_{3n+5}^3\to (\mathcal{C}, P_n)$ which implies
$\hat{R}(\mathcal{C}, P_n)\le 9n + O(1)$, a decent improvement over the previously known upper bound of $31n$.
\begin{itemize}
\item \tbf{Case 1} ($\up(0) = BRR$). 
If $\up(1) = RRR$, then edges 
$02, 03, 12$ and $13$ are all red and so  $(0213)$ would form a red $C_4$, a contradiction. Thus $\up(1)$ must contain at least one $B$ and we can take $P_B = 01$ which satisfies $r(P_B)=1$.

\item \tbf{Case 2} ($\up(0) = RBR$). 
If $\up(2)$ contains a $B$, then we could take $P_B = 02$ which has $r(P_B) = 1/2$. Otherwise we can assume $\up(2) = RRR$.  In this case edge $12$ must be blue, otherwise $(0123)$ forms a red cycle. Edge $13$ must be blue, otherwise $(013)$ forms a red cycle. Thus we may take $P_B=  0213$ which has $r(P_B) = 1$. Note that $\up(3)$ contains a $B$ as depicted in Figure \ref{fig:up-rrr} (with $v=2$).
\item \tbf{Case 3} ($\up(0) = RRB$). If $\up(3)$ contains a $B$, then we can take $P_B = 03$. Otherwise we can assume $\up(3) = RRR.$ If edge $23$ is red, then the red graph on vertices $\set{0,1,2,3,4,5,6}$ forms a tree, and so any uncolored edges must be blue. So in this case we may take $P_B = 0314$. Else we may suppose that edge $23$ is blue. One of the edges $24$ or $25$ must be blue, otherwise $(2435)$ is a red cycle. So then we can take $P_B = 0324$ or $P_B =0325$.
  
\item \tbf{Case 4}  ($\up(0) = BBR$).
One of the edges $12, 23$ or  $13$ must be blue, otherwise $(123)$ is a red cycle. So then we can take $P_B = 01$ or $P_B = 02$. 
\item \tbf{Case 5}  ($\up(0) = BRB$). One of the edges $13, 14$ or $34$ must be blue, otherwise $(134)$ is a red cycle. So then we can take $P_B = 01$ or $P_B = 03$.
\item \tbf{Case 6} ($\up(0) = RBB$). One of the edges $23, 24$ or $34$ must be blue, otherwise $(234)$ is a red cycle. So then we can take $P_B = 02$ or $P_B=03$.
\item \tbf{Case 7} ($\up(0) = BBB$). One of the edges $12, 23$ or  $13$ must be blue, otherwise $(123)$ is a red cycle. So then we can take $P_B = 01$ or $P_B = 02$. 
\end{itemize}

\subsection{Proof of Lemma \ref{lem:mainlem}: density $4/7$}
\begin{proof}[Proof of Lemma \ref{lem:mainlem}]
 The proof is essentially a more intricate version of the one that appears above. Cases 3 and 6 are much more involved than the other cases so the reader may wish to read those last. 
We provide \texttt{python} code at the url \url{http://msuweb.montclair.edu/~bald/research.html} which can help with the verification of this proof.
\begin{itemize}
\item \tbf{Case 1} ($\up(0) = BRR$)

If $\up(1) = RRR$, then edges 
$02, 03, 12$ and $13$ are all red and so  $(0213)$ would form a red $C_4$, a contradiction. Thus $\up(1)$ must contain at least one $B$ and we can take $P_B = 01$ which satisfies $r(P_B)=1$.

\item \tbf{Case 2} ($\up(0) = RBR$)

Suppose edge $12$ is red. Then $023$ is a blue path since edge $23$ must be blue  (else $(0123)$ is a red cycle).  If $\up(3) = RRR$, then  edge $14$ must be blue (else $(0143)$ is a red cycle), and so we can take $P_B = 02314$ since $4$ has blue up-neighbors 5 and 6 and $r(P_B) = 1.$ Otherwise $\up(3)$ contains a $B$ and we can take  $P_B = 023$ which satisfies $r(P_B) = 2/3.$

Now, suppose edge $12$ is blue. In this case, $0213$ is a blue path (edge $13$ must be blue otherwise $(013)$ is a red cycle). If $\up(3)$ contains a $B$, then we may take $P_B = 0213$. Otherwise $\up(3) = RRR$. In this case, edges $14$, $45$ and $56$ are all blue. Thus we can take $P_B = 02145$ where $r(P_B) = 4/5$.


\item\tbf{Case 3}  ($\up(0) = RRB$)  

Suppose edge $23$ is red.  Then edge $13$ must be blue (else $(0132)$ is a red cycle) and so $031$ is a blue path. 

\begin{itemize}
\item[] If edge $14$ were red, then edge $24$ must be  blue (else $(0142)$ is a red cycle) and so $03124$ is a blue path. If $\up(4)$ contains a $B$, then we may take $P_B = 03124$.  If $\up(4)=RRR$, then $320145$ is a red path, and so any other edge among these vertices must be blue.  Thus we may take $P_B = 03425$ since vertex $6$ is a blue up-neighbor of vertex $5$ and $r(P_B) = 4/5$.
\item[]  If edge $14$ were blue, then $0314$ is a blue path.
 If $\up(4)$ contains a $B$, then we may take $P_B = 0314$ which has $r(P_B) = 3/4$. 
Otherwise, suppose $\up(4) = RRR$ (which recall implies that vertices $5,6$ and $7$ form a blue triangle). 
\begin{itemize}
\item[] If edge $24$ were red, then edges $34$ and $25$ must be blue (else we have red cycles $(234)$ or $(245)$ respectively). Thus we may take $P_B = 034125$ since vertex $6$ is a blue up-neighbor of vertex $5$. 
\item[] So we assume edge $24$ is blue. If edge $35$ is red then edge $25$ must be blue (else $(235)$ is a red cycle). Thus we may again take $P_B= 031425$. So assume that edge $35$ is blue.  In this case, we have $03567$ is a blue path.  Now if $\up(7)$ contains a $B$, when we may take $P_B = 03567$ which has $r(P_B) = 4/7$ (this specific case is illustrated in Figure \ref{fig:rrb} just as an example). Otherwise if $\up(7) = RRR$, then edge 68 is blue (else $(4687)$ is a red $C_4$).  In this case we may take $P_B = 0357689$ which has $r(P_B) = 2/3$.
\begin{figure}[h]
\centering
\includegraphics[scale=.7]{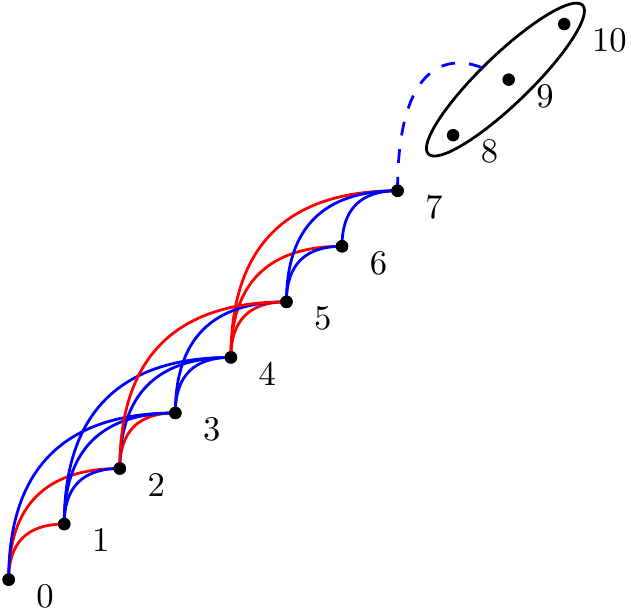}
\caption{An illustration of the situation when the proof has led us to the assumptions  $\up(0) = RRB$, edge $23$ is red, 14 is blue, $\up(4)=RRR$, 24 is blue, 35 is blue and $\up(7)$ contains a $B$. In this case, we take $P_B = 03567$ which has $r(P_B) = 4/7$.}
\label{fig:rrb}
\end{figure}

\end{itemize}

\end{itemize}

Now we assume edge $23$ is blue. Then $0321$ forms a blue path.  
\begin{itemize}
\item[] Suppose edge $14$ is red. Then edge $24$ is blue (else $(0142)$ is a red cycle) and so $0324$ is a blue path. If $\up(4)$ contains a $B$, then we may take $P_B = 0324$. So suppose that  $\up(4) = RRR$. Then $(567)$ is a blue triangle and edge $25$ must be blue (else $(01452)$ is a red cycle) and so we may take $P_B = 03256$ which has $r(P_B) = 2/3$. 

\item[] Now suppose edge $14$ is blue. If $\up(4)$ contains  a $B$, then we may take $P_B=03214$. Else suppose $\up(4)=RRR$ so that $(567)$ forms a blue triangle. If edge $25$ is blue, then we may take  $P_B = 03256$ which has $r(P_B) = 2/3$. So suppose edge $25$ is red.

\begin{itemize}
\item[]  If edge $35$ is blue, then 03567 is blue path. If $\up(7)$ contains a $B$, then we may take $P_B = 03567$ with $r(P_B) = 4/7$. Otherwise suppose $\up(7) = RRR$. Then we may take $P_B = 0357689$ which has $r(P_B) = 2/3$.

\item[]  So suppose edge  $35$ is red. Then edge $36$ is blue (else $(3546)$ is a red $C_4$). So $03657$ is a blue path. If $\up(7)$ contains a $B$, then we take $P_B = 03657$  with $r(P_B)=4/7$. Otherwise suppose $\up(7)=RRR$ and so  edge $58$ is blue (else $(4578)$ is a red $C_4$). So we may take  $P_B = 0367589$ which has $r(P_B) =2/3$. 

\end{itemize}

 \end{itemize}

\item\tbf{Case 4}  ($\up(0) = BBR$) 

If $\up(1)$ contains a $B$, then we may take $P_B = 01$. Otherwise suppose $\up(1) = RRR$. In this case, $(234)$ is a blue $C_3$ and so we may take $P_B = 023$ which has $r(P_B) = 2/3.$

\item\tbf{Case 5}  ($\up(0) = BRB$)

If $\up(1)$ contains a $B$, then we may take $P_B = 01$. Otherwise suppose $\up(1)=RRR$ so that $(234)$ is a blue $C_3$. Then $0324$ is a blue path. If $\up(4)$ contains a $B$, then we may take $P_B = 0324$ which has $r(P_B) = 3/4$. Otherwise suppose $\up(4) = RRR$ so that $(567)$ is a blue $C_3$ and so that edge $25$ is blue (else $(1254)$ is a red $C_4$). Then we may take $P_B= 034256$ which has $r(P_B) = 5/6$.

\item\tbf{Case 6}  ($\up(0) = RBB$)  

Suppose edge $23$ is blue. If $\up(3)$ contains a $B$, then we may take $P_B=023$ which has $r(P_B)=2/3$. Otherwise suppose $\up(3)=RRR$.  Then edges $24$ and $25$ cannot both be red (else $(2435)$ is a red $C_4$). If edge $24$ blue, then we may take $P_B=03245$ which has $r(P_B)=4/5$. If edge $25$ is blue, then  we may take $P_B = 0325$ which has $r(P_B)=3/5$. 

So suppose edge $23$ is red. Then edges $12$ and $13$ cannot both be red (else $(123)$ is a red $C_3$).

First suppose both edges $12$ and $13$ are both blue.
Then $0213$ is a blue path. If $\up(3)$ contains a $B$, then we may take $P_B=0213$. Otherwise $\up(3)=RRR$ in which case edge $24$ is blue (else $(234)$ is a red $C_3$) and so we may take  $P_B =  031245$.

Now suppose that exactly one of $12$ or $13$ is blue and the other is red. Denote the blue edge as $1\a$ and the red edge as $1\b$, where $\a,\b\in\set{2,3}, \a\neq\b$. Notice that the edge $\a\b$ is red since this is the edge $23$.
\begin{itemize}
\item[] Suppose edge $14$ is blue. If $\up(4)$ contains a $B$, then we may take $P_B=0\a14$. Otherwise suppose $\up(4) = RRR$. 
\begin{itemize}
\item[] If edge $\b4$ is red, then the red graph on vertices $\set{0, \ldots, 7}$ forms a tree, and so any other edge on these vertices must be blue. In particular, edges $25, 35$ and $36$ are blue and so we may take $P_B=02536$ which has $r(P_B) = 2/3$.
\item[] So suppose edge $\b4$ is blue. If edge $\b5$ is blue, then we may take $P_B = 0\a14\b56$. If edge $\b5$ is red, then again, the red graph on vertices $\set{0,\ldots, 7}$ forms a tree, and so any other edge is blue. In particular, edge $\alpha 5$ is blue and so we may take $P_B= 0\b41\a56$.

\end{itemize}
\item[] Suppose edge $14$ is red. Then edges $24$ and $34$ are both blue since the red graph on vertices $\set{0,\ldots, 4}$ forms a tree. 

\begin{itemize}
\item[] Suppose edge $25$ is red.  Then the red graph on vertices $\set{0,\ldots, 5}$ forms a tree and so any other edge on these vertices must be blue. In particular edges $35$ and $45$ are blue. So $02435$ forms a blue path. If $\up(5)$ contains a $B$, then we may take $P_B=02435$. Otherwise suppose $\up(5)=RRR$, in which case we may take $P_B = 0245367$.
\item[] Suppose edge 25 is blue. If $\up(5)$ contains a $B$, then we may take $P_B= 02435$. Otherwise suppose that $\up(5)=RRR.$ If edge $46$ is red, then the red graph on $\set{0,\ldots, 8}$ forms a tree and so we may take $P_B=  0245367.$ If edge $46$ is blue, then we may take $P_B = 03467$ which has $r(P_B) = 4/7$.
\end{itemize}
\end{itemize}

\item\tbf{Case 7}  ($\up(0) = BBB$) 

 If $\up(1)$ contains a $B$, then we may take $P_B = 01$. Otherwise suppose $\up(1)=RRR$ in which case we may take $P_B = 023$.
\end{itemize}

\end{proof}

\section{Computer assisted improvement}\label{sec: comp}
With the use of a computer program (written in \texttt{python}, making use of the \texttt{networkx} package,  and made available at the url\footnote{A \texttt{.txt} file containing the output of the program is also available.} \url{http://msuweb.montclair.edu/~bald/research.html}) we have a proof of the following lemma which finds a higher density path than Lemma \ref{lem:mainlem}.

\begin{lemma}\label{lem:comp}
Let $Q = P_{43}$ on vertex set $\set{0,1,\ldots 42}$ and let $H = Q^3$. Suppose that $H$ has been 2-colored with no red cycles from $\mathcal{C}_{\le8}$. Further suppose that in $H$, $\up(0)\not\in\set{ RRR, RRB}$. Then there is a $k\in \set{1,\ldots 39}$ such that $H[\set{0,\ldots, k}]$ contains a blue path $P_B$ with endpoints $0$ and $k$ such that $up(k)\not\in\set{ RRR, RRB}$ and
\[r(P_B):=\frac{\abs{V(P_B)\cap \set{1,\ldots, k}}}{k} \ge \frac{19}{25}.\]
\end{lemma}

Using this improved density of $19/25=.76$, Theorem \ref{thm:comp} follows just as Theorem \ref{thm:main} followed from Lemma \ref{lem:mainlem}. The algorithm proceeds much as our proof of Lemma \ref{lem:mainlem} proceeds. Suppose $\up(0),\ldots, \up(k-1)$ have been assigned and one finds neither a red cycle nor a blue path of the desired ratio ending at $k-1$. Then we iterate through all $8$ possibilities for $\up(k)$, again searching for a red cycle or a high density blue path (ending at $k$) and deepening the recursion when neither is found.
 In order to cut down on cases checked, we forced the program to avoid the most work intensive ``Case 3'', hence the requirement $\up(0), \up(k)\not\in \set{RRR, RRB}$. Note that any coloring of $P_N^3$ with no red cycles must satisfy $\set{\up(0), \up(1)}\not\subseteq\set{ RRR, RRB}$ and so this is an okay assumption.  
 As a demonstration of the growth of complexity, we mention that the output of the program which verifies a density of $4/7$ (i.e. equivalent to the proof of Lemma \ref{lem:mainlem}) is a \texttt{.txt} file of size 85 KB. The file which verifies the density of $3/4$ has size 1.7 MB and the file which verifies the density of $19/25$ has size 34 MB. As discussed in Section \ref{sec:ub}, the best density one could hope for in $P_N^3$ is $7/9 \approx 0.7777$. Due to our proof method (stitching together segments), it is unlikely that our program (as currently written) will be able to prove the exact bound of $7/9$; one can color the portion near vertex 0 `badly' in a way that lowers the overall density of the segment.

\section{Lower Bound}\label{sec:lb}
In this section we prove Theorem \ref{thm:lb} by improving on the ideas which appear in \cite{DKP}.  
The following reduction lemma essentially appears as a lemma  in \cite{BD}. In that paper, the lemma concerns avoidance monochromatic paths in both colors rather than cycles in red and a path in blue. However, the proof is almost identical, so we have decided to omit it. This lemma allows us to concentrate on graphs with minimum degree at least 3.
\begin{lemma}\label{lem:mindeg3}
Let  $n$ be a positive integer with $n\geq 6$. If every connected graph with at most $m$ edges and minimum degree at least $3$  has a $2$-coloring such that the red graph is acyclic, and every blue path has order less than $n-2$, then every graph with at most $m$ edges has a $2$-coloring such that the red graph is acyclic and  every blue path has order less than $n$.
\end{lemma}

We also make use of the following lemma which shows how to find a forest in a bounded degree graph whose removal creates many vertices of degree 0 or 1 (thus unsuitable for paths in the remaining graph).
\begin{lemma}\label{lem:deg4n9}
Suppose $G$ is connected and has $n$ vertices and maximum degree $\Delta$. Then G contains a forest $F$ and disjoint subsets $A_0, A_1 \subseteq V(G)$ such that
\begin{enumerate}
\item $A_0\cup A_1$ is an independent set
\item $d_F(v) = d_G(v)$ for all $v\in A_0$
\item $d_F(v) \ge  d_G(v)-1$ for all $v\in A_1$
\item $|A_0| + \frac12 |A_1|\ge \gamma_\Delta n $ where \[\gamma_\Delta =\of{\frac{1}{\Delta^2 + \Delta + 2} + \frac{3}{2(\Delta^2 + 2\Delta  +3)}}. \]
\end{enumerate}

\end{lemma}
\begin{proof}

We greedily build the forest $F$ and maintain disjoint sets $A_0, A_1, X, Y$. Throughout,  $X=N(A_0\cup A_1)$ and $Y = V(G)\setminus(A_0\cup A_1\cup X)$, and so there are no edges between $Y$ and $A_0\cup A_1$. Initialize $A_0, A_1, X, F = \emptyset$ and $Y=V(G).$ 

We start with \emph{Phase 1}. Begin by adding an arbitrary vertex to $A_0$, removing it from $Y$ and updating $X$. At each subsequent step of Phase 1, we  look for a vertex $v\in Y$ such that $|N(v)\cap X| \le 1$. If such a vertex $v$ exists, we add $v$ to $A_0$, add all of $v$'s incident edges to $F$, and include all of $v$'s neighbors in $X$. When no such vertex $v$ exists, then Phase 1 ends. At the end of Phase 1,  every vertex in $Y$ has at least $2$ neighbors in $X$ and every vertex in $X$ has at most $(\Delta-1)$ neighbors in $Y$ (since each vertex in $X$ has a neighbor in $A_0$), so 
$2|Y|\le e(X,Y)\le (\Delta-1)|X|$ and also $|X| \le \Delta|A_0|$.
So at the end of Phase 1 \[n = |A_0| + |X| + |Y| \le |A_0| + \Delta|A_0| + \frac{\Delta-1}{2}\Delta|A_0| = \of{\frac{\Delta^2+\Delta +2}{2}}|A_0|\] so $|A_0| \ge \frac{2}{\Delta^2 + \Delta + 2}n$.

In \emph{Phase 2} we add vertices to $A_1$ which have $|N(v)\cap X|\le2$. If there is a vertex with $|N(v)\cap X|\le1$, we handle it as above.
If no such $v$ exists, then we next look for a vertex $v\in Y$ such that $|N(v)\cap X|=2$. In this case, we move $v$ to $A_1$, we add to $F$, any edges incident to $v$ and not $X$. Of the two edges incident to both $v$ and $X$, we arbitrarily choose one to add to $F$. If no such $v$ exists, we terminate the process. At the end of Phase 2, every vertex in $Y$ has at least $3$ neighbors in $X$.

By construction, one can observe that $A_0\cup A_1$ remains independent since we only add vertices from $Y$. Also by construction, $F$ remains a forest  and the degree conditions in (ii) and (iii) are met. It remains to show that at the end of the process, $|A_0| + \frac12 |A_1| \ge \gamma_\Delta n.$

Note that $|X| \le \Delta|A_0\cup A_1|$ and that at the end of Phase 2, we have $3|Y| \le e(X,Y)\le (\Delta-1)|X|$. Thus at termination, we have $|Y|\le \frac{\Delta -1 }{3}|X|$ and so
\begin{align*}n = |A_0\cup A_1| + |X| +|Y|&\le |A_0\cup A_1| + \frac{\Delta+2}{3}|X| \\
&\le |A_0\cup A_1| +  \frac{\Delta+2}{3}\cdot\Delta |A_0\cup A_1|  \\
&=\frac{\Delta^2 + 2\Delta + 3}{3}  |A_0\cup A_1| 
\end{align*}
and so $|A_0|+|A_1| \ge \frac{3}{\Delta^2 + 2\Delta + 3} n$.
To finish, we observe
\begin{align*}
|A_0|+\frac12 |A_1| & = \frac12 |A_0| + \frac12 (|A_0|+|A_1|) \\
&\ge \of{ \frac12 \frac{2}{\Delta^2 + \Delta + 2} + \frac12  \frac{3}{\Delta^2 + 2\Delta + 3} }n = \gamma_\Delta n.
\end{align*}

\end{proof}

\begin{proof}[Proof of Theorem \ref{thm:lb}]
Suppose $G=(V,E)$ is connected, has $e = (2+\alpha)n$ edges and $G\to (\mathcal{C}, P_n)$. In light of Lemma \ref{lem:mindeg3}, we also assume that $\delta(G)\ge 3.$ We note that technically, by using Lemma \ref{lem:mindeg3}, we should now change our goal to finding a coloring such that the red graph is acyclic and every blue path is of order less than  $n-2$. For readability, we continue to forbid paths of order $n$ and mention that the $O(1)$ in the statement of Theorem \ref{thm:lb} takes care of the issue.
We may assume that $G$ has $N = (1+\beta)n$ vertices where $\beta \le \alpha$ (else we may take any spanning tree, color it red and note that there are too few remaining edges to have a blue path of order $n$). 

Let $X$ be the set of vertices of degree at least $4$.
Then $|X| \ge 2n-N$. To see this, note that by the assumption $G\to (\mathcal{C}, P_n)$, $G$ must have a path of order $n$ and we may color its edges red (which is acyclic in red). Then the uncolored edges must have a path of order $n$ (otherwise we could color them all blue). Thus we have two edge-disjoint paths $P_1, P_2$ on vertex sets $A_1, A_2$, each of size $n$, and any vertex in $A_1\cap A_2$ has degree at least 4. Thus we have $|X| \ge |A_1\cap A_2| = |A_1|+|A_2| - |A_1\cup A_2| \ge 2n- N$.

Let $B$ be the set of vertices of degree at least $d+1$ (we will end up taking $d=5$).
Then \begin{align*}
(4+2\alpha)n = 2e = \sum_v d(v)&\ge (d+1)|B| + 4(|X|-|B|) + 3(N - |X|) \\
& = (d-3)|B| + |X| + 3N \\
&\ge (d-3)|B| + 2n + 2N
\end{align*}
and so rearranging, we have
\[ |B| \le \frac{1}{d-3}\of{(2+2\alpha)n - 2N} = \frac{2}{d-3}\of{ (1+\alpha)n - N  } = \frac{2}{d-3}(\alpha-\beta)n.   \]
Let 
\[\gamma_d:=\of{\frac{1}{d^2 + d + 2} + \frac{3}{2(d^2 + 2d  +3)}}.\]
Note that $G[V\setminus B]$ has maximum degree $d$ and so we may apply Lemma \ref{lem:deg4n9} to each component of $G$ in order to find a forest $F$ and sets $A_0, A_1$
with 
\[|A_0| + \frac{1}{2}|A_1| \ge  \gamma_d\cdot\of{N-|B|}. \]
We color all edges in $F$ with red, complete this forest to a red tree in $G$ and then color the remaining edges in $G$ with blue.  
Let $R = V\setminus(A_0 \cup A_1\cup B)$. So every vertex in $A_0$ has only red edges to $R$ and every vertex in $A_1$ has at most one blue edge to $R$. 
Suppose $P = (v_1,v_2, \ldots, v_k)$ is a blue path. Note that if $v_i \in A_0$ for some $1< i < k$, then $v_{i-1}$ and $v_{i+1}$ must both be in $B$. Also, if $v_i \in A_0$ for some $1< i < k$, then at least one of $v_{i-1}$ and $v_{i+1}$ is in $B$. For $X\in \set{A_0, A_1, B, R}$, let $X'  = V(P)\cap X.$ 
So if we let $e_P(A_0\cup A_1, B)$  count the number of edges in $P$ with one end in $A_0\cup A_1$ and the other end in $B$, we have $2|A'_0| + |A'_1| -2\le e_P(A_0\cup A_1, B)\le  2|B'|$, and so $|A'_0| + |A_1'|  \le |B'| + \frac12 |A_1'|  +1$. We then have
\begin{align*}
|V(P)|  &= |R'| + |A'_0| + |A'_1| + |B'| \\
&\le |R|+ \frac12|A_1'|+ 2|B'| +1 \\
&\le N - |A_0|-|A_1| -|B|  + \frac12 |A_1'| + 2|B'| +1\\
&\le N - |A_0| - \frac12 |A_1| + |B| +1.
\end{align*}
We see that if $N - (|A_0| +\frac12 |A_1|) +|B| < n-1$, then there is no blue path of order $n$.

\begin{align*}
\frac{1}{n-1}\of{N-(|A_0| +\frac12 |A_1|) +|B|}&\le \frac{1}{n-1}\of{N -  \gamma_d\of{N-|B|} + |B|} \\
& = \frac{1}{n-1}\of{(1-\gamma_d) N + (1+\gamma_d)|B| } \\
&\le  (1-\gamma_d) (1+\beta) +   (1+\gamma_d) \cdot \frac{2}{d-3}(\alpha-\beta) + O(1/n)\\
&= (1 - \gamma_d) + \beta\of{1-\gamma_d - \frac{2}{d-3}(1+\gamma_d)} \\
&\qquad+ \alpha\of{\frac{2}{d-3}}(1+\gamma_d) + O(1/n).
\end{align*}
We set \[f(\alpha, \beta, d):= (1 - \gamma_d) + \beta\of{1-\gamma_d - \frac{2}{d-3}(1+\gamma_d)}+ \alpha\of{\frac{2}{d-3}}(1+\gamma_d).\]
This function is decreasing in $\beta$ for $d=4,5$ and increasing in $\beta$ for $d\ge 6$.
When $d=5$, we may maximize this function by setting $\beta=0$, and in this case we get 
\[f(\alpha, 0, 5) = \frac{651}{608}\alpha + \frac{565}{608}.\]
So we have that $N - (|A_0| +\frac12 |A_1|) +|B| < n-1$ whenever $\alpha < \frac{43}{651} - \Omega(1/n)$. One can check that using $d=4, 6$ yields the bounds $\alpha < 5/109$ and $\alpha < 39/709$ (recalling that for $d=6$, one must set $\beta = \alpha$ when maximizing) both of which are worse than $43/651$. For all $d\ge7$, the bound is also worse. 
\end{proof}

\section{Concluding Remarks}\label{sec:conc}
In this paper we have considered the size Ramsey number for the family of cycles versus a path of order $n$. In contrast to many recent results on size Ramsey numbers of paths and cycles, we use a non-random construction. This, however, is due to the fact that the question considered included forbidden short cycles. We note in passing that by considering the third power of a cycle $C_N^{3}$ with $N = \frac{25}{19} n + O(1)$, our proof easily implies that \[\hat{R}(\mathcal{C}_{\le 8}, \mathcal{C}_{\ge n}) \le 3.947n \] where $\mathcal{C}_{\ge n}$ is the family of all cycles of length at least $n$.

The most obvious open problem is to close the gap between the lower bound of $2.066n$ and the upper bound of $3.947n$.  It is possible that there is a nice  proof that every two coloring of $P_N^3$ contains a blue path of density $7/9$, but we were unable to find one. 

\subsection*{Acknowledgements} The first author would like to thank Louis DeBiasio for enlightening discussions on the topic.

\end{document}